\documentclass[a4paper]{amsart}
\usepackage{amssymb,graphicx,rotating,hyperref}
\usepackage[all]{xy}
\xyoption{import}

\def\d{{\rm d}}
\def\sech{\,{\rm sech}}
\def\t{{\mathfrak t}}
\def\g{{\mathfrak g}}
\def\R{{\bf R}}
\def\Z{{\bf Z}}
\def\ad#1{{\rm ad}_{#1}}
\def\Ad#1{{\rm Ad}_{#1}}
\def\vf{{\mathcal X}}

\def\ie{{\it i.e.\ }}
\def\im{{\rm{Im}\, }}
\def\rmt#1{{\rmfamily\tiny #1}}
\def\H{{\sf H}}

\newtheorem{theorem}{Theorem}[section]
\newtheorem{thm}{Theorem}[section]
\newtheorem{lemma}[thm]{Lemma}

\begin{document}

\title{Positive-Entropy Geodesic Flows on Nilmanifolds}
\author{Leo T. Butler, Vassili Gelfreich}
\date{\today}

\begin{abstract} Let $T_n$ be the nilpotent group of real $n
\times n$ upper-triangular matrices with $1$s on the diagonal. The
hamiltonian flow of a left-invariant hamiltonian on $T^*T_n$ naturally
reduces to the Euler flow on $\t_n^*$, the dual of $\t_n = {\rm
Lie}(T_n)$. This paper shows that the Euler flows of the standard
riemannian and subriemannian structures of $T_4$ have transverse
homoclinic points on all regular coadjoint orbits. As a corollary,
left-invariant riemannian metrics with positive topological entropy
are constructed on all quotients $D \backslash T_n$ where $D$ is a
discrete subgroup of $T_n$ and $n \geq 4$.
\end{abstract}

\keywords{Geodesic flows, entropy, nilmanifolds, nonintegrability,
subriemmannian geometry} 

\maketitle

\section{Introduction}
\label{int}

Let $\Sigma$ be a nilmanifold, \ie homogeneous space of a connected
nilpotent Lie group $G$. Each homogeneous riemannian metric on $G$
induces a locally-homogeneous metric on $\Sigma$. These riemannian
geometries, which will be called {\it left-invariant}, are of interest
in both geometry and dynamics. A basis question is

\medskip
\noindent
{\bf Question A}: {\it Which left-invariant geodesic flows on a
compact nilmanifold have zero topological entropy?}
\medskip

A mistaken answer to question A appears in Theorem 3
of~\cite{Manning}. In~\cite{Butler:2003a}, the first author showed
that on $2$-step nilmanifolds, all left-invariant geodesic flows have
zero entropy. In~\cite{Butler:2003b}, metrics on compact quotients of
the $3$-step nilpotent Lie group $T_4 \oplus T_3$ are constructed
whose geodesic flows have positive topological entropy. The paper also
speculated that the standard geodesic flow on $T_4$ also had such
horseshoes. Montgomery, Shapiro and Stolin~\cite{Montgomeryetal}
investigated the standard {\em subriemannian} geodesic flow on
$T_4$~\cite{Butler:2003a}; they showed that it reduces to the
Yang-Mills hamiltonian flow which is known to be algebraically
non-integrable~\cite{Zig1,Zig2}.

Let us state the first result of the present paper. The Lie algebra of
$T_4$, $\t_4$, has the standard basis consisting of those $4\times 4$
matrices $X_{ij}$ with a unit in the $i$-th row and $j$-th column, $i
< j$, and zeros everywhere else. A quadratic hamiltonian $h : \t_4^*
\to \R$ is {\em diagonal} if it is expressed as $h(p) = \sum_{i<j}
a_{ij} \langle p,X_{ij} \rangle^2$ for some constants $a_{ij}$. The
standard riemannian metric has $a_{ij}=1$ for all $i,j$; the standard
Carnot (subriemannian) metric has $a_{12}=a_{23}=a_{34}=1$ and all other
coefficients zero.

\begin{theorem} \label{thm:main}
If $h : \t_4^* \to \R$ is a diagonal hamiltonian with
$a_{12}a_{13}a_{23}a_{34} \neq 0$ and $a_{13}a_{34}=a_{12}a_{24}$,
then for all but at most countably many regular coadjoint orbits in
$\t_4^*$, the Euler vector field of $h$ has a horseshoe. In
particular, the Euler vector field of the standard riemannian metric
(resp. sub-riemannian metric with $a_{13} \neq 0$) is analytically
non-integrable.
\end{theorem}

\medskip
The condition that $a_{13}a_{34}=a_{12}a_{24}$ is only a device to
simplify the proof: all nearby hamiltonians also have a horseshoe.  We
also show, by means of a numerical computation of an integral (see
Table \ref{tab:1} in section \ref{sec:deg}), that when $a_{13}=0$,
the conclusions of Theorem \ref{thm:main} hold. This shows that the
subriemannian geodesic flow of Montgomery, Shapiro and Stolin is
real-analytically non-integrable. Related numerical computations
(figure \ref{fig:1} in section \ref{sec:deg}) suggest that the
hamiltonian $h$ has a horseshoe on {\em every} regular coadjoint orbit
provided only that $a_{12}a_{23}a_{34} \neq 0$ and
$a_{13}a_{34}=a_{12}a_{24}$. Let us formulate a corollary to Theorem
\ref{thm:main}: Let $D < T_n$ be a discrete subgroup of $T_n$, $\Sigma
= D \backslash T_n$ and $S\Sigma$ is the unit sphere bundle.

\medskip
\begin{theorem} \label{thm:cor1}
If $n \geq 4$, then there is a left-invariant geodesic flow $\phi_t :
S\Sigma \to S\Sigma$ such that $h_{top}(\phi_1) > 0$.
\end{theorem}

\medskip
It appears likely that {\em all} left-invariant geodesic flows on
$S\Sigma$ have positive topological entropy and are non-integrable
with smooth integrals. 

\medskip

Theorem \ref{thm:cor1} is interesting from a riemannian point of
view. Let $(M,g)$ be a smooth ($C^{\infty}$) riemannian manifold, and
$\phi_t : SM \to SM$ the geodesic flow of $g$. For each $T > 0$ and
$p,q \in M$, let $n_T(p,q)$ denote the number of distinct geodesics of
length no more than $T$ that join $p$ to $q$. Ma\~{n}\'e~\cite{Man}
showed that if $M$ is compact then $h_{top}(\phi_1 | SM) = \lim_{T \to
\infty} T^{-1} \log \int_{M \times M}\ n_T(p,q)\, dpdq$. Thus, for the
geodesic flows constructed here, for generic points $p$ and $q$ on a
compact quotient of $T_n$, $n_T(p,q)$ grows exponentially fast. In
constrast, Karidi showed that the volume growth on the universal cover
of these manifolds is polynomial of degree
$\frac{1}{6}n(n^2-1)$~\cite{Kar}.

Theorem \ref{thm:main} is proved by reducing the hamiltonian flow of
$\phi_t$ on $T^* \Sigma$ to a hamiltonian flow on the coadjoint orbits
of ${\rm Lie(}T_4{\rm )}^*$. In the appropriate coordinate system, the
reduced hamiltonian is a small perturbation of a hamiltonian on $\R^4$
that is the sum of an unforced Duffing hamiltonian and a forced linear
system whose solutions can be expressed in terms of Legendre
functions. The Poincar\'e-Melnikov technique developed
in~\cite{MarsdenHolmes:1983a,Robinson:1988a} for autonomous
Hamiltonian systems is adapted here to show that a suspended Smale
horseshoe appears in the perturbed hamiltonian flow.

\section{The construction on Lie($T_4$)${}^*$}

In this section, we will first recall a number of key facts about
geodesic flows and left-invariant hamiltonian systems on the cotangent
bundle of a Lie group; for more details, see~\cite{GS}. We will
then reduce the equations of motion of a left-invariant geodesic flow
on $T^* T_4$ to the equations of motion of a hamiltonian system on
$T^*\R^2$.

\subsection{Poisson geometry of left-invariant hamiltonians}

A {\em Poisson manifold} is a smooth manifold $M$ such
that $C^{\infty}(M)$ is equipped with a skew-symmetric bracket $\{,\}$
that makes $( C^{\infty}(M), \{,\} )$ into a Lie algebra of
derivations of $C^{\infty}(M)$. The centre of $( C^{\infty}(M), \{,\}
)$ is traditionally called the set of {\em Casimirs}. If $f$ is a
Casimir then $X_f \equiv 0$ and $f$ is a first integral of all
hamiltonian vector fields. If the set of Casimirs of $( C^{\infty}(M),
\{,\} )$ are the constant functions, then we say that $(
C^{\infty}(M), \{,\} )$ is a {\em symplectic} manifold. In this case,
the Poisson bracket naturally induces a closed, non-degenerate skew
$2$-form on $M$ which is called a symplectic structure. We will say
that a smooth map $f : M \to N$ is a {\em Poisson map} if $f^* : (
C^{\infty}(N), \{,\}_N ) \to ( C^{\infty}(M), \{,\}_M )$ is a Lie
algebra homomorphism.

The most basic example of a Poisson manifold that is also symplectic
is provided by $T^* \R = \{ (a,A)\ : a,A \in \R\}$ equipped with the
Poisson bracket satisfying $\{a,A\}_{T^* \R} = 1$.

The dual space of a Lie algebra gives an example of a Poisson manifold
that is not (in general) a symplectic manifold. Let $\g$ be a
finite-dimensional real Lie algebra and let $\g^*$ be the dual vector
space of $\g$. $T^*_p \g^*$ is identified with $\g$ for all $p \in
\g^*$. The Poisson bracket on $\g^*$ is defined for all $f,h \in
C^{\infty}(\g^*)$ and $p \in \g^*$ by
\begin{equation} \label{eq:pb}
\{f,h\}(p) := - \langle p, [ df_p, dh_p ] \rangle,
\end{equation}
where $\langle \cdot,\cdot \rangle : \g^* \times \g \to \R$ is the
natural pairing. Recall that for $\xi \in \g$, $\ad{\xi}^* : \g^* \to
\g^*$ is the linear map defined by $\langle \ad{\xi}^*p, \eta \rangle
= -\langle p, [\xi,\eta] \rangle$. $\ad{}^* : \g \to gl(\g^*)$ is the
representation contragredient to the adjoint representation. For any
$h \in C^{\infty}(\g^*)$, the hamiltonian vector field $E_h = \{\cdot
, h\}$ equals $-\ad{dh_p}^* p$. The standard example of a hamiltonian
vector field is obtained from a positive-definite linear map $\phi :
\g^* \to \g$ by setting $h(p) = \frac{1}{2} \langle p, \phi(p)
\rangle$, in which case $E_h(p) = -\ad{\phi(p)}^* p$.

Let $G$ be a connected Lie group whose Lie algebra is $\g$. The
adjoint representation of $G$ on $\g$, $\Ad{g}\xi =
\frac{d}{dt}|_{t=0}\, g \exp(t\xi) g^{-1}$, induces the coadjoint
representation $\langle \Ad{g}^* p, \xi \rangle = \langle p,
\Ad{g^{-1}} \xi \rangle$ for all $p \in \g^*$, $g \in G$ and $\xi \in
\g$. As each vector field $p \to \ad{\xi}^* p$ is hamiltonian on
$\g^*$, with linear hamiltonian $h_{\xi}(p) = - \langle p, \xi
\rangle$, the coadjoint action of $G$ on $\g^*$ preserves the Poisson
bracket. The orbits of the coadjoint action are called the coadjoint
orbits. Each coadjoint orbit is a homogeneous $G$-space, and {\em
every} hamiltonian vector field on $\g^*$ is tangent to each coadjoint
orbit. For this reason, the Poisson bracket $\{,\}_{\g^*}$ restricts
to each coadjoint orbit, and is non-degenerate on each coadjoint
orbit. Thus, the coadjoint orbits are naturally symplectic
manifolds. A Casimir is necessarily constant on each coadjoint orbit,
and in many cases (as in this paper) each coadjoint orbit is the
common level set of all Casimirs.

The Poisson bracket on $\g^*$ also arises in a natural way from the
Poisson bracket on $T^* G$. The group $G$ acts from the left on $T^*
G$, and this action preserves the Poisson structure
$\{,\}_{T^*G}$. The set of smooth left-invariant functions
$C^{\infty}(T^* G)^G$ is therefore a Lie subalgebra of $C^{\infty}(T^*
G)$ with respect to $\{,\}_{T^*G}$. This subalgebra is naturally
identified with $( C^{\infty}(\g^*), \{,\}_{\g^*})$ as follows: the
left-trivialization of $T^* G = G \times \g^*$ induces the projection
map $r : T^* G \to \g^*$ onto the second factor; $r^* C^{\infty}(\g^*)
= C^{\infty}(T^* G)^G$ and $r^*$ is a Lie algebra monomorphism.

The hamiltonian flow of a left-invariant hamiltonian $H$ on $T^* G$
therefore has the equations of motion:
\begin{equation}
X_H(g,p) = \left\{ 
\begin{array}{ccc} 
\dot{g} &=& T_e L_g dh(p),\\
\dot{p} &=& -\ad{dh(p)}^*\ p,
\end{array}
\right.
\end{equation}
Note that $\d r(X_H) = E_h$, where $h \in C^{\infty}(\g^*)$ satisfies
$r^* h = H$. The vector field $E_h$ is called the {\em Euler vector
field}. It is a natural reduction of $X_H$ by $G$. If $h(p) =
\frac{1}{2} \langle p, \phi(p) \rangle$ for a positive-definite linear
map $\phi : \g^* \to \g$ then $H$ is induced by a left-invariant
metric on $T^* G$ and $X_H$ is the geodesic vector field.

Finally, if $D < G$ is a discrete subgroup, then $T^* \Sigma = \Sigma
\times \g^*$ where $\Sigma = D \backslash G$. The projection map $r :
T^* G \to \g^*$ is naturally left-invariant, so it factors through to
a map $r_o : T^* \Sigma \to \g^*$. If $H_o = r_o^*h$ for some $h \in
C^{\infty}(\g^*)$ then $Tr_o(X_{H_o}) = E_h$. Thus, the hamiltonian
flow of a left-invariant hamiltonian on $T^* \Sigma$ always projects
to a hamiltonian flow on $\g^*$.

\subsection{Poisson geometry of $T^*T_4$}

Let $\t_4$ denote the Lie algebra of $T_4$, so
$$
\t_4 = \left\{ {
\left[ \begin{array}{cccc}
0 & x & z & w\\
0 & 0 & y & u\\
0 & 0 & 0 & v\\
0 & 0 & 0 & 0
\end{array} \right] } 
\ :\ u,v,w,x,y,z \in \R \right\}.
$$ For each $a \in \{u,v,w,x,y,z\}$, let $A \in \t_4$ be the
element obtained by setting $a$ equal to one and all other
coefficients equal to zero. Then $\{U,V,W,X,Y,Z\}$ is a basis of
$\t_4$ whose commutation relations given by: $[X,Y]=Z$,$[Y,V]=U$,
$[X,U]=W$, $[Z,V]=W$, and all others are trivial or obtained by
skew-symmetry. 

Let $p_{a} : \g^* \to \R{}$ be the linear function given by $p_{a}(p)
= \langle p, A \rangle$ for $A \in \g$ and all $p \in \g^*$. From the
definition of the Poisson bracket on $\t_4^*$ ({\em
c.f.}~Eq.~\ref{eq:pb}), along with the commutation relations, we
conclude that: $$\{p_x,p_y\}=-p_z,\ \{p_y,p_v\}=-p_u,\
\{p_x,p_u\}=-p_w,\ \{p_z,p_v\}=-p_w.$$

There are two independent Casimirs of $\t_4^*$ are $K_1(p) = p_w$, $K_2(p) =
p_w p_y - p_z p_u$. Let $K : \t_4^* \to \R^2$ be
defined by $K = (K_1,K_2)$. The level sets of $K$ are the
coadjoint orbits of $T_4$'s action on $\t_4^*$ and will be denoted by
${\mathcal O}_{k}$, where $k=(k_1,k_2)$. We will say that
${\mathcal O}_{k}$ is a {\em regular coadjoint orbit} if $k_1 k_2
\neq 0$.

\begin{lemma}  \label{lem:ok}
Let ${\mathcal O}_k$ be a regular coadjoint orbit. Then ${\mathcal
O}_{k}$ is symplectomorphic to $(T^* \R)^{2}$ equipped with its
canonical symplectic structure.
\end{lemma}

\begin{proof} 
The Poisson bracket on $\t_4^*$ restricts to ${\mathcal O}_k$. We will
denote the restricted bracket by $\{,\}_k$. Let $(T^* \R)^{2} = \{
(a,A,b,B)\ :\ a,A,b,B \in \R \}$. The canonical Poisson bracket, which
we denote by $[,]$, satisfies $[a,A]=1$, $[b,B]=1$ and all other
brackets are zero.

Let $\lambda$ and $\mu$ be two non-zero parameters (the parameters are
included because we will further transform coordinate systems). The
map $f_{k}(p) = (a,A,b,B)$ defined by
$$a= -\lambda p_x,\ A=(k_1\lambda)^{-1} p_u,\ b= - \mu p_v,\ B=
(k_1\mu)^{-1} p_z
$$
is a diffeomorphism of ${\mathcal O}_{k}$ onto $T^*\R^2$. Indeed,
$f_k$ is clearly smooth. And $g_k(a,A,b,B) = (p_u, \ldots, p_z)$
defined by 
$$p_v = -\mu^{-1}b,\ p_w = k_1,\ p_x = -\lambda^{-1}a,\ p_y =
(k_2+k_1^2 \lambda \mu AB)/k_1,\ {\rm and\ } p_z = k_1 \mu B,$$
satisfies $K \circ g_k = k$ and $f_k \circ g_k = id$, $g_k \circ f_k =
id$. Since $g_k$ is an algebraic map, it is smooth, and so we see
${\mathcal O}_k$ is diffeomorphic to $T^* \R^2$.

The commutation relations for the Poisson bracket on $\t_4^*$ allow us
to compute that $\{f_k^* a, f_k^* A\}_{k} = \{f_k^* b, f_k^* B\}_{k} =
1$ and all other Poisson brackets are zero. It follows that $f_k^* : (
C^{\infty}( T^* \R^2) , [,] ) \to ( C^{\infty}({\mathcal O}_k),
\{,\}_{k})$ is a Lie algebra isomorphism. Hence $f_k : {\mathcal O}_k
\to T^* \R^2$ is a symplectomorphism.
\end{proof}

\subsection{The hamiltonians}

Let $a_{ij} > 0$ be constants such that $a_{13}a_{34}=a_{12}a_{24}$ and let
\begin{equation} \label{eq:h}
4H(p) = a_{12} p_x^2 + a_{23} p_y^2 + a_{13} p_z^2 + a_{24} p_u^2 +
a_{34} p_v^2 + a_{14} p_w^2
\end{equation}
Since the vector field $E_{H}$ is unaffected by the addition of a
Casimir, the term $a_{14} p_w^2$ can be ignored. 

Let us introduce a symplectic change of variables on $T^* \R^2$: $A =
\frac{1}{\sqrt{2}} (X-Y)$, $B = \frac{1}{\sqrt{2}} (X+Y)$, $a =
\frac{1}{\sqrt{2}} (x-y)$, $b = \frac{1}{\sqrt{2}} (x+y)$, $z = c$ and
$Z = C$. Because $a_{13}a_{34}=a_{12}a_{24}$, there exists unique
$\lambda,\mu > 0$ so that $0 = a_{34} \mu^{-2} - a_{12} \lambda^{-2} =
a_{13} \mu^2 - a_{24} \lambda^2$, and $a_{12} \lambda^{-2} + a_{34}
\mu^{-2} = 1$. Indeed, we can choose $\lambda^2 = 2a_{12}$ and $\mu^2
= 2a_{34}$. Then:
\begin{equation} \label{eq:h1}
2{\bf H}_{k} = (x^2 - \xi X^2 + \nu X^4) + (y^2 + \omega Y^2 + \nu
Y^4 - 2 \nu X^2 Y^2), 
\end{equation}
where $\xi = - ( a_{13} a_{34} k_1^2 + a_{23} k_2 \sqrt{a_{12} a_{34}})
$, $\omega = \xi + 2 a_{13} a_{34} k_1^2 = a_{13} a_{34} k_1^2 - a_{23} k_2
\sqrt{a_{12} a_{34}}$ and $\nu = a_{12} a_{23} a_{34} k_1^2$. Note that we
can write $\omega=\xi + 2c\nu$ where $c= \frac{a_{13}}{a_{12}a_{23}}$.

\begin{lemma} \label{lem:h2}
The Euler vector field of $H$ on the regular coadjoint orbit
${\mathcal O}_k$ (equation \ref{eq:h}) is a time change of the hamiltonian
vector field of 
\begin{equation} \label{eq:h2}
2{\bf H} = x^2 + \left(X^2-\frac{1}{2}\right)^2 + y^2 + \alpha^2 Y^2 +
Y^4 - 2X^2Y^2
\end{equation}
on $T^* \R^2$, where $\alpha^2 = 1 + 2c\nu^{\frac{1}{3}}$.
\end{lemma}

\begin{proof}
Define the coordinate change $g_{\nu}(x,X,y,Y) = (ax, a^{-1}X, ay,
a^{-1}Y)$ where $a=\nu^{\frac{1}{6}}$. Then $g^*({\bf H}_k) = a^2 {\bf
H}$.
\end{proof}

\begin{lemma} \label{lem:h3}
For all $\epsilon > 0$, the hamiltonian flow of ${\bf H}$ (equation
\ref{eq:h2}) is conjugate to the flow of the vector field
\begin{equation} \label{eq:vf}
\vf_{\epsilon} = \left\{ 
\begin{array}{lclclcl}
\dot{X} &=& x, & \hspace{2mm} & \dot{Y} &=& y,\\
\dot{x} &=& X - 2X^3 + 2\epsilon XY^2, & & \dot{y} &=& \left[ -\alpha^2
  + 2 X^2 \right]\, Y + 2\epsilon Y^3.
\end{array} 
\right.
\end{equation}
\end{lemma}

\begin{proof}
Introduce the coordinate transformation $h_{\epsilon}(x,X,y,Y) =
(x,X,\sqrt{\epsilon}y,\sqrt{\epsilon}Y)$. 
\end{proof}

\medskip
\noindent{\bf Remark.} It is clear from (\ref{eq:vf}) that the vector
field $\vf_{\epsilon}$ depends on the parameter $\alpha$. Inspection
of the formula for $\alpha$ (lemma \ref{lem:h2} and immediately above)
shows that $\alpha$ is identically unity when the coefficient $a_{13}$
vanishes. This is the case for the standard subriemannian metric,
where $a_{12}=a_{23}=a_{34}=1$ and the other coefficients vanish. Rather than
specializing to $\alpha = 1$, we have elected to carry $\alpha$
through our analysis. The rationale for this will be apparent in
section \ref{ssec:mel}.
\medskip

\section{Analysis of $\vf_{\epsilon}$} \label{sec:anax}
For $\epsilon=0$, the map $h_{\epsilon}$ is singular. However, the
vector field $\vf_0$ is well-defined. We will show that $\vf_0$ has a
normally hyperbolic invariant manifold $S$ whose stable and unstable
manifolds coincide, and that this manifold $S$ persists for
$\epsilon>0$ but the stable and unstable manifolds
$W^{\pm}_{\epsilon}(S)$ no longer coincide. This implies that the
Euler vector field $E_H | {\mathcal O}_k$ has transverse homoclinic
points for all regular coadjoint orbits.

\subsection{The normally hyperbolic manifold $S$}
Inspection of the vector field $\vf_{\epsilon}$ shows that the set 
$$S = \{ (x,X,y,Y)\ :\ x=X=0 \},$$
is invariant for all $\epsilon$. One sees that for $\epsilon=0$, the
vector field is
$$\vf_0 =  \left\{ 
\begin{array}{lclclcl}
\dot{X} &=& x, & \hspace{2mm} & \dot{Y} &=& y,\\
\dot{x} &=& X + O(X^3), & & \dot{y} &=& \left[ -\alpha^2 + 2X^2 \right]\, Y,
\end{array} 
\right. 
$$
which shows that $S$ is normally hyperbolic. Therefore $S$ is
normally hyperbolic for all $\epsilon$ sufficiently small. Since
$\vf_{\epsilon}$ is conjugate to the same vector field for all
non-zero $\epsilon$, one concludes that $S$ is a normally hyperbolic
manifold for {\em all} $\epsilon$.

\subsection{The stable and unstable manifolds of $S$} \label{ssec:wpmS}
The function $h = x^2 + (X^2-\frac{1}{2})^2$ is a first integral of
$\vf_0$. The set $h^{-1}(\frac{1}{4})$ is the stable and unstable
manifold of $S$, which we denote by $W^{\pm}_0(S)$. On $W^{\pm}_0(S)
- S$, the flow of $\vf_0$ satisfies 
\begin{align} \label{eq:separatrix}
\left\{
\begin{array}{lcllcl}
X &=& \pm \sech(t+t_0), & x &=& \mp
\tanh(t+t_0)\sech(t+t_0)^2,\\
Y &=&
c_0 Y_0(t+t_0) + c_1Y_1(t+t_0), & y &=& \dot{Y},
\end{array}
\right.
\end{align}
where $X(0)=\pm \sech(t_0), x(0)=\mp \tanh(t_0)\sech(t_0)^2$ and $Y_j$
solves the initial-value problem
\begin{align} \label{eq:deY}
\left\{
\begin{array}{rclrcl}
\ddot{Y} + \left[ \alpha^2 - 2\sech(t)^2 \right]\, Y &=& 0, \hspace{20mm}
(*)\\
& \\
Y(0) = 1-j,\quad \dot{Y}(0)=j
\end{array}
\right.
\end{align}
while $Y(0)=c_0Y_0(t_0) + c_1Y_1(t_0)$,
$y(0)=c_0\dot{Y}_0(t_0) + c_1\dot{Y}_1(t_0)$. The solutions $Y_j$ are
chosen so that they are even ($j=0$) and odd ($j=1$) functions of
time.

\subsection{The Melnikov function} \label{ssec:mel}
To determine if the flow of $\vf_{\epsilon}$ has transverse homoclinic
points for non-zero $\epsilon$, we appeal to the following theorem.

\begin{thm} \label{thm:melint}
Let $\varphi_{\epsilon} : M \times \R \to M$ be a complete, smooth
flow that depends smoothly on $\epsilon$. Assume that $\varphi_0$
possesses a normally hyperbolic, invariant manifold $S \subset M$, and
that there is a smooth function $h : M \to \R$ such that
\begin{enumerate}

\item the stable and unstable manifolds of $S$ coincide and equal
  $h^{-1}(0)$;
\item $\d h$ does not vanish on $ W^{\pm}_0(S) - S.$

\end{enumerate}
Then, for all sufficiently small non-zero $\epsilon$,
$\varphi_{\epsilon}$ possesses a normally hyperbolic invariant
manifold $S_{\epsilon}$ and the local stable and unstable manifolds of
$S_{\epsilon}$ ($W^+_{\epsilon}(S)$ and $W^-_{\epsilon}(S)$,
respectively) can be written as the graph of a function
$s_{\epsilon}^{\pm} : W^{\pm}_0(S) \to W^{\pm}_{\epsilon}(S)$. The
splitting distance, defined for $p \in W^{\pm}_0(S)$ by
$s_{\epsilon}(p) = h \circ s^+_{\epsilon}(p) - h \circ
s^-_{\epsilon}(p)$, is a smooth function of $\epsilon$ and
$s_{\epsilon}(p) = \epsilon m(p) + O(\epsilon^2)$ where
\begin{equation} \label{eq:melint}
m(p) = \int_{t \in \R} \langle \d h, {\mathcal Y} \rangle \circ
\varphi_0^t(p)\, \, \d t,
\end{equation}
and ${\mathcal Y} = \displaystyle \left. \frac{\partial\ }{\partial
\epsilon} \frac{\partial\ }{\partial t} \right|_{t=\epsilon=0}\,
\varphi^t_{\epsilon} $.
\end{thm}

\begin{proof}
The proof of this theorem is a standard application of invariant
manifold theory plus an adaptation of the proof of the Melnikov
formula~\cite{MarsdenHolmes:1983a,Robinson:1988a,Gruendler}.
\end{proof}

\medskip
\noindent{\bf Remark.} If $m$ changes sign, then, for all $\epsilon
\neq 0$ sufficiently small, the perturbed stable and unstable
manifolds intersect but do not coincide. In our case, a result of
Burns and Weiss implies that the topological entropy is non-zero. If
$m$ has a non-degenerate zero, then the implicit function theorem
implies that, for all $\epsilon \neq 0$ sufficiently small,
$W^+_{\epsilon}(S)$ has a transverse intersection with
$W^-_{\epsilon}(S)$. Note that the intersection of the surface $S$
with a constant energy level is a periodic orbit. Therefore each
trajectory in the intersection is doubly asymptotic to a periodic
orbit in $S$.
\bigskip

For the flow defined by $\vf_{\epsilon}$, we have that
\begin{equation} \label{eq:y}
{\mathcal Y} = 
\left\{ 
\begin{array}{lclclcl}
\dot{X} &=& 0, & \hspace{2mm} & \dot{Y} &=& 0,\\
\dot{x} &=& 2 XY^2, & & \dot{y} &=& 2 Y^3.
\end{array} 
\right.
\end{equation}
Whence
\begin{equation} \label{eq:dhy}
\langle \d h, {\mathcal Y} \rangle = 4 xXY^2,
\end{equation}
since $h = x^2 + (X-\frac{1}{2})^2$. The equations in
(\ref{eq:separatrix}) imply that the Melnikov function is
\begin{equation} \label{eq:melfn}
m(p) = c_0c_1 \times  \int_{\tau \in \R} -4\, \tanh(\tau)\, \sech(\tau)^2\
Y_0(\tau)\, Y_1(\tau)\ \d\tau.
\end{equation}

\medskip
\noindent{\bf Remarks.} (1) The formula for the Melnikov integral
(\ref{eq:melfn}) appears to be a function on $S$ not
$W^{\pm}_0(S)-S$. This does not contradict Theorem
(\ref{thm:melint}). Inspection of the integral (equation
\ref{eq:melint}) shows that $m(\varphi_0^s(p)) = m(p)$ for all $s$ and
$p$. The coordinate system on $W^{\pm}_0(S)-S$ determined by
(\ref{eq:separatrix}) uses time along the flow as one coordinate
($t_0$), so only the other two coordinates, $c_0$ and $c_1$, ought to
appear in the Melnikov function. (2) If we write $m(p) = 2 c_0c_1 \times
I$, where
\begin{equation} \label{eq:I}
I=\int_{0}^{\infty} -4\, \tanh(\tau)\, \sech(\tau)^2\ Y_0(\tau)\,
Y_1(\tau)\ \d\tau,
\end{equation} then $m$ has
non-degenerate zeros along $\{ c_0=0,c_1\neq0\ {\rm or}\
c_1=0,c_0\neq0 \}$, provided that $I\neq 0$.
\medskip

\subsection{The Legendre functions and $I$}
Substitution of $z=\tanh(t)$ transforms the differential equation
(\ref{eq:deY}*) into the Legendre differential equation
\begin{equation} \label{eq:legendre}
(1-z^2)Y'' -2zY' + \left( \nu(\nu+1) - \frac{\mu^2}{1-z^2}
  \right) Y = 0,
\end{equation}
where $\mu=i\alpha$, $\nu=-\frac{1}{2}+\frac{\sqrt{-7}}{2}$ and
$'=\frac{\d\ }{\d z}$. The integral $I$ (equation \ref{eq:I}) is
transformed to
\begin{equation} \label{eq:Imod}
I = \int_{0}^1 z\, U_0(z)\, U_1(z)\, \d z,  
\end{equation}
where $U_j(z) = Y_j(t)$. 

\subsection{The Melnikov function is non-zero: $I\neq0$}
\label{ssec:melfn0}
For the remainder of this note, $q~:~[0,\infty)~\to~\R$ is a
continuously differentiable function such that $\lim_{t\to\infty} q(t)
= 0$ and $\alpha > 0$ is a fixed positive number. The function
$z=z(t)$ is assumed to solve
\begin{equation} \label{eq:dez}
\ddot{z} + [\alpha^2 - q(t)]\, z = 0.
\end{equation}
In analogy with the integral $I$ in equation (\ref{eq:I}), define an
integral
\begin{equation} \label{eq:Iz}
I = \int_{0}^{\infty} \dot{q}(t)\, z_0(t)\, z_1(t)\, \d t
\end{equation}
$I$ is implicitly a function of $\alpha$; one wants to prove that
$I$ can vanish at most countably many times. 
 
Let us first prove the following.
\begin{lemma} \label{lem:bdd}
If $z$ solves equation (\ref{eq:dez}), then $z$ is bounded with
bounded derivative.
\end{lemma}
\begin{proof}
Define $H = \frac{1}{2} (\alpha^2 z^2 + \dot{z}^2)$. From (\ref{eq:dez}) one
computes that $\dot{H} = qz\dot{z}$. Integrating by parts yields $H =
C_0 + \frac{1}{2} q(t)z(t)^2 + \int_0^t - \frac{1}{2} \dot{q}(s)\, z(s)^2\
\d s$, where $C_0$ is a constant that depends only on $z(0)$ and
$\dot{z}(0)$. Thus
\begin{equation} \label{eq:zbd}
\frac{1}{2} \left[\, \alpha^2 - q(t) \right]\, z(t)^2 \leq C_0 +
\int_0^t - \frac{1}{2}\, \dot{q}(s)\, z(s)^2\ \d s.
\end{equation}
Since $q(t) \to 0$ as $t \to \infty$, there is a $T \geq 0$ such that
$\alpha^2 - q(t) \geq \frac{1}{2}\alpha^2$ for all $t \geq
T$. Therefore, equation (\ref{eq:zbd}) implies that there is a constant
$C_1$ such that for all $t \geq 0$
\begin{equation} \label{eq:zbd2}
z(t)^2 \leq C_1 + \frac{4}{\alpha^2} \times \int_0^t -\dot{q}(s)\, z(s)^2\
\d s
\end{equation}
This is a Gronwall inequality for $u=z^2$. Thus
$$z(t)^2 \leq C_1 \, \exp(-\frac{4}{\alpha^2} q(t) )$$
for all $t \geq 0$. Since $q$ is continuous and converges to $0$ at
infinity, it is bounded. Therefore, $z$ is bounded.

To prove that $w=\dot{z}$ is bounded, define $H=\frac{1}{2} (\alpha^2
w^2 + \dot{w}^2)$. Since $\ddot{w}+[\alpha^2 - q(t)]w = f$, $f=\dot{q}z$,
one computes that $\dot{H} = qw\dot{w} + f\dot{w}$. One can bound
$\int_0^t f(s)\dot{w}(s)\d s$ using that $\ddot{z}=\dot{w}$ and that
$q$ and $z$ are already bounded. One then obtains a Gronwall
inequality like (\ref{eq:zbd2}) for $w(t)^2$.
\end{proof}

\begin{lemma} \label{lem:zero}
If $z_0,z_1$ are solutions to equation (\ref{eq:dez}), then the limit
\begin{equation} \label{eq:limit}
I = -\lim_{t\to \infty} \left[ \dot{z}_0(t)\, \dot{z}_1(t)\, + \alpha^2
  z_0(t)\, z_1(t) \right]
\end{equation}
 exists and equals $W \alpha\cot(B)$ where the angle $B$ is defined
  below, and $W$ is the Wronskian of the solutions $z_0,z_1$.
\end{lemma}

\begin{proof}
Let $z$ be any solution of (\ref{eq:dez}). Let $t_n$ be the sequence
of zeros of $z(t)$, indexed in increasing order. Let $\phi_n = \alpha
t_n \bmod 2\pi$ and $a_n = \dot{z}(t_n)$. There is a sequence $n_k$
such that $\phi_{n_k} \to \phi \bmod 2\pi$ and $a_{n_k} \to a>0$. The
former follows by compactness of $\R/2\pi\Z$ and the latter because
$\dot{z}$ is bounded.

Since $q(t) \to 0$, the Sturm comparison theorem implies that the
sequence $\phi_n$ converges to $\phi$ and $t_{n+1}-t_n$ converges to
$\pi/\alpha$~\cite{Zettl}. One sees that positive and negative zeros
of $z$ must alternate for all $n$. Hence, without loss of generality,
one may assume that $\dot{z}(t_{2n}) \to a$ and $\dot{z}(t_{2n+1}) \to
-a$. The continuous dependence of solutions on initial data therefore
implies that $z^n(t) := z(t+\frac{2n\pi}{\alpha})$ converges in the
weak Whitney $C^1$ topology to $a \sin(\alpha t-\phi)$. In particular,
$z^n(t)$ converges uniformly to $a \sin(\alpha t-\phi)$ for $t \in
[0,4\pi/\alpha]$.

To apply these observations to the limit (\ref{eq:limit}), let $N$ be
sufficiently large so that the $n$-th and $n+1$-th zeros of both $z_0$
and $z_1$ are at most $\frac{2\pi}{\alpha}$ apart for all $n\geq
N$. Let $s\in [\frac{2n\pi}{\alpha},\frac{2(n+1)\pi}{\alpha}]$ and write
$s=t+\frac{n\pi}{\alpha}$ so that $t\in [0,\frac{2\pi}{\alpha}]$. Then
\begin{align*}
&\ \ \ \ |\dot{z}_0(s) \dot{z}_1(s) + \alpha^2 z_0(s)z_1(s) - \alpha^2 a_0a_1
\cos(\phi_1-\phi_0)|\\ 
&\leq |\dot{z}^n_0(t) \dot{z}^n_1(t) - \alpha^2
a_0a_1 \cos(\alpha t - \phi_0) \cos(\alpha t - \phi_1)| +\\
&\ \ \ \ \ \alpha^2 | z^n_0(t)z^n_1(t) - a_0a_1 \sin(\alpha t - \phi_0)
\sin(\alpha t - \phi_1)|.
\end{align*}
If $s \to \infty$, then $n\to \infty$. The above-mentioned uniform
convergence for $t\in [0,2\pi/\alpha]$ shows that the limit
(\ref{eq:limit}) exists and equals $A\cos(B)$ where $A=\alpha^2
a_0a_1$ and $B=\phi_1-\phi_0 \bmod 2\pi$.

On the other hand, the Wronskian $W$ of $z_0,z_1$ is
constant and
$$\xymatrix{W = z_0(t)\dot{z}_1(t) - z_1(t)\dot{z}_0(t)\ \ar[rr]^{t \to
  \infty} && \ \alpha a_0a_1 \sin(\phi_1-\phi_0),}$$ by the
  same argument as above. Therefore $I/W = \alpha \cot(B)$.
\end{proof}

\medskip
\noindent{\bf Remarks.} (1) The angle $B$ has the following
interpretation which emerges from the proof of lemma
(\ref{lem:zero}). The zeros of solutions to (\ref{eq:dez}) are
asymptotically $\pi/\alpha$ apart, and the zeros of linearly
independent solutions are interlaced. The angle $B$ is defined so that
$B/\alpha \bmod \pi/\alpha$ is asymptotically the time between
consecutive zeros of the linearly independent solutions.  Figure
\ref{fig:1}, left, plots $B$ as a function of $\alpha$ for the
solutions $z_j=Y_j$ to the initial-value problem (\ref{eq:deY}). One
expects that as $\alpha \to \infty$, the solutions should converge
quite quickly to $\cos$ and $\sin$, whence $B$ should approach
$\frac{\pi}{2}$. The figure captures this behaviour quite nicely.  (2)
The function $I = W\alpha \cot(B)$ from lemma~\ref{lem:zero} can be
computed numerically. The Sturm comparison theorem implies that the
$n$-th zero $t_n$ of a solution $z$ satisfies $\pi/\alpha < t_{n+1} -
t_n < \pi/\alpha \times (1 + q/\alpha^2)$ if $|q(t)| < \alpha^2/2$ for
all $t > t_n$. If $q$ goes to zero sufficiently fast, one can
numerically compute the first several zeros and obtain a reasonably
accurate estimate of $B$. Figure \ref{fig:1}, right, shows the graph
of $I$ for $q=2\sech(t)^2$.
\medskip

\begin{lemma} \label{lem:I0}
If $z_0, z_1$ are solutions to (\ref{eq:dez}) such that $\dot{z}_0$
and $z_1$ vanish at $t=0$, then the integral
$$
I = \int_{0}^{\infty} \dot{q}(t)\, z_0(t)\, z_1(t)\, \d t \leqno({\rm \ref{eq:Iz}})
$$
exists and equals $W \alpha\cot(B)$ where the angle $B$ is described
in Lemma \ref{lem:zero}, and $W$ is the Wronskian of the solutions
$z_0,z_1$.
\end{lemma}

\begin{proof}
By lemma \ref{lem:bdd}, both solutions are bounded, so one can apply
integration by parts to the integral. This yields
\begin{align*}
I &= q(0) z_0(0) z_1(0) - \int_{0}^{\infty} q(t)\, \left[
  \dot{z}_0(t)\, z_1(t)\, + z_0(t)\, \dot{z}_1(t)\, \right]\, \d t,\\
&= -\int_{0}^{\infty} q(t)\, \left[
  \dot{z}_0(t)\, z_1(t)\, + z_0(t)\, \dot{z}_1(t)\, \right]\, \d t
\end{align*}
since $z_1$ vanishes at $t=0$.

From (\ref{eq:dez}), it is known that $q(t)z_0(t) = \ddot{z}_0(t)+\alpha^2
  z_0(t)$ and similarly for $z_1$. Therefore
\begin{align*}
I &= -\int_{0}^{\infty} \frac{\d\ }{\d t}\,
  \left[ \dot{z}_0(t)\, \dot{z}_1(t)\, + \alpha^2 z_0(t)\, z_1(t)\,
  \right]\, \d t,\\
&= - \lim_{t\to \infty} \left[ \dot{z}_0(t)\,
    \dot{z}_1(t)\, + \alpha^2 z_0(t)\, z_1(t) \right], \hspace{5mm}
  \textrm{since}\ \dot{z}_0(0)=0=z_1(0),\\
&= W \alpha \cot(B) \hspace{45mm} {\rm by\ lemma\ (\ref{lem:zero}).}
\end{align*}
\end{proof}

\begin{lemma} \label{lem:nov}
Assume that there exists $C, \lambda > 0$ such that $|q(t)| < C
e^{\lambda t}$ for all $t > 0$. Then the integral $I=I(\alpha)$ is a
holomorphic function of $\alpha$ on the strip $|\im \alpha| < \lambda$
about the real line.

Consequently, if $q$ is an even, monotone function, then $I$ vanishes
countably many times at most.
\end{lemma}

\begin{proof}
A solution $z=z(t;\alpha)$ to (\ref{eq:dez}) is a holomorphic function
of $\alpha$ for each fixed $t$ \cite{Zettl}. For large $t$ and $|\im
\alpha| < \lambda$, the solution $z=z(t;\alpha)$ is equal to
$\cos(\alpha t + \phi)$ plus a term that grows slower than $e^{\lambda
  t}$. This implies, by the residue formula, that $I=I(\alpha)$ is
holomorphic provided that $|\im \alpha| < \lambda$.

When $\alpha=0$ and $q$ is even, the even and odd solutions to
(\ref{eq:dez}) do not change sign. Therefore, if $q$ is monotone, then
the integrand defining $I(0)$ does not change sign, so $I(0)\neq
0$. Thus, $I$ can vanish at most countably many times on the strip
$|\im \alpha| < \lambda$.
\end{proof}

\begin{proof}[Theorem \ref{thm:main}]
If $a_{13} \neq 0$ -- whence $c \neq 0$ in Lemma \ref{lem:h2} --, then
lemma \ref{lem:nov} shows that the hamiltonian flow of $H$ (equation
\ref{eq:h}) on all but countably many coadjoint orbits in $\t^*_4$ has
a horseshoe.  This proves the main result, Theorem \ref{thm:main}.
\end{proof}

\section{The degenerate case when $\alpha \equiv 1$} \label{sec:deg} 
If $a_{13} = 0$, as occurs for the Carnot subriemannian metric of
\cite{Montgomeryetal}, then $\alpha \equiv 1$ and lemma \ref{lem:nov}
cannot be applied. We investigate two distinct ways to address this
problem. The first is direct and numerical; the second leads to some
further insight into the integral $I$. 

\subsection{Numerical evidence} \label{ssec:num}
In this case, figure \ref{fig:1} indicates that $I(1)$ is
approximately $-2.75$. Table \ref{tab:1} shows the results of a
numerical computation of $I(1)$ with varying step sizes. It is clear
from this table that $I(1)=-2.76$ to two decimal places. 

To estimate the error in the computations, one uses the fact that the
differential equations (\ref{eq:dez}) are hamiltonian with the
hamiltonian
\begin{equation} \label{eq:Hz}
\H = \frac{1}{2}p^2 + \frac{1}{2}\left[ \alpha^2 - q(\tau)  \right]\,
z^2 + u, 
\end{equation}
where $p,z$ and $u,\tau$ are canonically conjugate variables (along
solutions, $\tau = \tau_0+t$, so it is a pseudo-time). Since $\H$ is
preserved along solutions to (\ref{eq:dez}), the maximum deviation of
$\H$ along a numerical solution provides an estimate of the upper
bound of the error in the solutions $z_0, z_1$.

\subsection{Qualitative evidence} \label{ssec:qual} 
As explained in the Remark in subsection \ref{ssec:melfn0}, one may
compute $I$ as a function of $\alpha$ by computing the phase angle
$B$. Figure \ref{fig:1} graphs $B$ and $I$ versus $\alpha$. This
figure shows that $I(1)$ does not vanish.

\begin{sidewaystable}
\centering
\caption{The numerical calculation of $I$ with $\alpha=1$. }  \label{tab:1}
\begin{tabular}{|llllllll|} 
\hline\hline $h$ & $I$ & $\H_{0,min}$ & $\H_{0,max}$ & $\H_{1,min}$ &
$\H_{1,max}$ & $\H_{0,max}-\H_{0,min}$ & $\H_{1,max}-\H_{1,min}$
\\ \hline\hline

\rmt{ 0.5 } & \rmt{ -2.76812630 } & \rmt{ -0.5 } & \rmt{ -0.49025150 } & \rmt{ 0.49833857 } & \rmt{ 0.51067514 } & \rmt{ 0.00974849 } & \rmt{ 0.01233657}\\
\rmt{ 0.25 } & \rmt{ -2.76366763 } & \rmt{ -0.5 } & \rmt{ -0.49944022 } & \rmt{ 0.49992738 } & \rmt{ 0.50063022 } & \rmt{ 0.00055977 } & \rmt{ 0.00070283}\\
\rmt{ 0.125 } & \rmt{ -2.76340793 } & \rmt{ -0.5 } & \rmt{ -0.49996571 } & \rmt{ 0.49999582 } & \rmt{ 0.50003878 } & \rmt{ 3.42847126 $\times 10^{-5}$  } & \rmt{ 4.29572646 $\times 10^{-5}$ }\\
\rmt{ 0.0625 } & \rmt{ -2.76339200 } & \rmt{ -0.5 } & \rmt{ -0.49999786 } & \rmt{ 0.49999974 } & \rmt{ 0.50000241 } & \rmt{ 2.13207876 $\times 10^{-6}$  } & \rmt{ 2.67223559 $\times 10^{-6}$ }\\
\rmt{ 0.03125 } & \rmt{ -2.76339101 } & \rmt{ -0.5 } & \rmt{ -0.49999986 } & \rmt{ 0.49999998 } & \rmt{ 0.50000015 } & \rmt{ 1.33088170 $\times 10^{-7}$  } & \rmt{ 1.66835298 $\times 10^{-7}$ }\\
\rmt{ 0.015625 } & \rmt{ -2.76339095 } & \rmt{ -0.5 } & \rmt{ -0.49999999 } & \rmt{ 0.49999999 } & \rmt{ 0.50000000 } & \rmt{ 8.31541421 $\times 10^{-9}$  } & \rmt{ 1.04238692 $\times 10^{-8}$ }\\
\rmt{ 0.0078125 } & \rmt{ -2.76339094 } &
\rmt{ -0.5 } & \rmt{ -0.49999999 } &
\rmt{ 0.49999999 } & \rmt{
  0.50000000 } & \rmt{ 5.19672801 $\times 10^{-10}$ } &
\rmt{ 6.51456639 $\times 10^{-10}$} \\\hline
\multicolumn{8}{p{16cm}}{$ $}\\
\multicolumn{8}{p{16cm}}{Solutions to the hamiltonian equations of $\H$ are computed
  with the Forest-Ruth $4$-th order symplectic integrator
  \cite{SanzSerna} and initial conditions
  $z(0)=j,\dot{z}(0)=1-j,\tau(0)=0,u(0)=0$ for $j=0,1$. The maximum
  (resp. minimum) value of $\H$ along the $j$-th numerical solution
  over the interval $[0,35]$ is indicated by $\H_{j,max}$
  (resp. $\H_{j,min}$). The integral $I$ is computed by
\[
\hspace{-4cm} I^h = h \times \sum_{i=0}^N \dot{q}(t_i)\, z^h_0(t_i)\, z^h_1(t_i)
\]
where $z^h_j$ is the computed solution with step size $h$, $N=35/h$,
$t_i = i \times h$ and $q(t) = 2\sech(t)^2$.
Files at \url{http://www.maths.ed.ac.uk/~lbutler/t4.html}}
\end{tabular}
\end{sidewaystable}

{
\def\figa{{
\def\orbitpica{\resizebox{5cm}{!}{\includegraphics{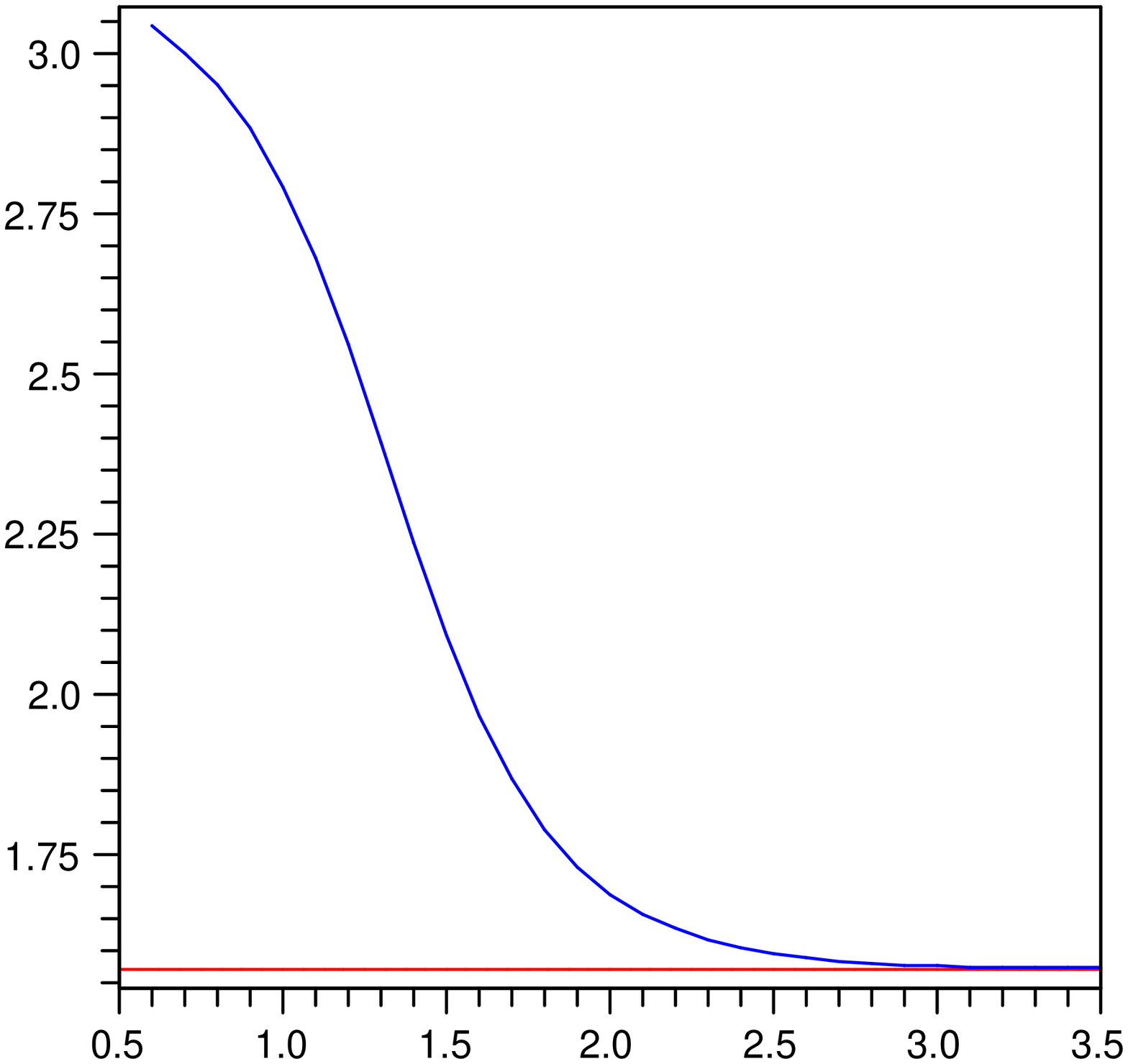}}}
\xy
\xyimport(4,4){\orbitpica}
,(2.,3.)*!L\txt{$B(\alpha)$}*=0{},{\ar-(.5,.5)}
,(1.2,1.1)*!L\txt{$\frac{\pi}{2}$}*=0{},{\ar-(.1,.4)}
,(2.2,.1)*!L\txt{$\alpha$}*=0{}
,(-0.1,2.5)*!L\txt{$B$}*=0{}
\endxy
}}
\def\figb{{
\def\orbitpicb{\resizebox{5cm}{!}{\includegraphics{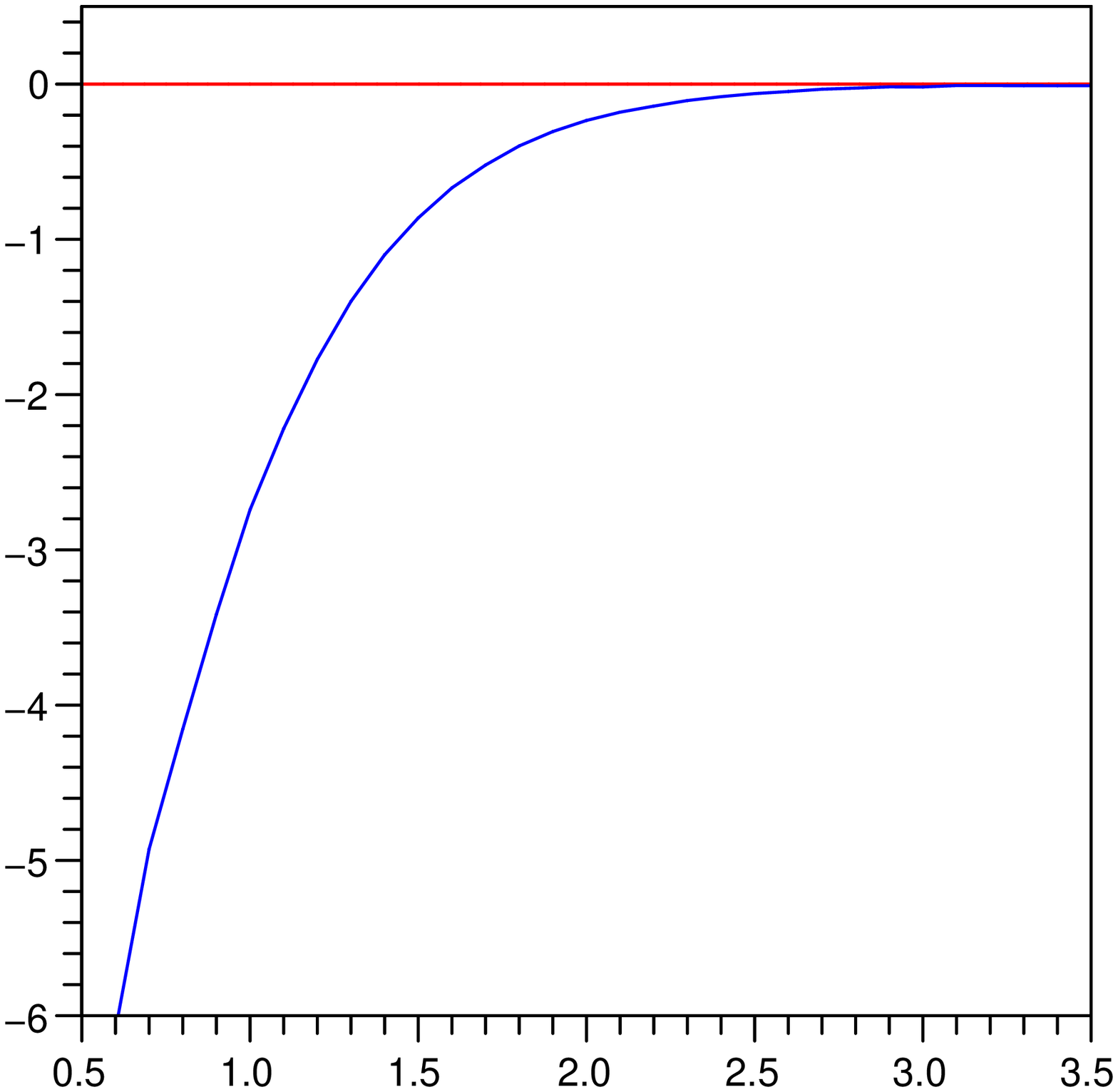}}}
\xy
\xyimport(4,4){\orbitpicb}
,(2.,3.)*!L\txt{$\alpha\cot(B)$}*=0{},{\ar-(0.5,0.0)}
,(2.2,.0)*!L\txt{$\alpha$}*=0{}
,(-0.1,2.5)*!L\txt{$I$}*=0{}
\endxy
}}
\begin{center}
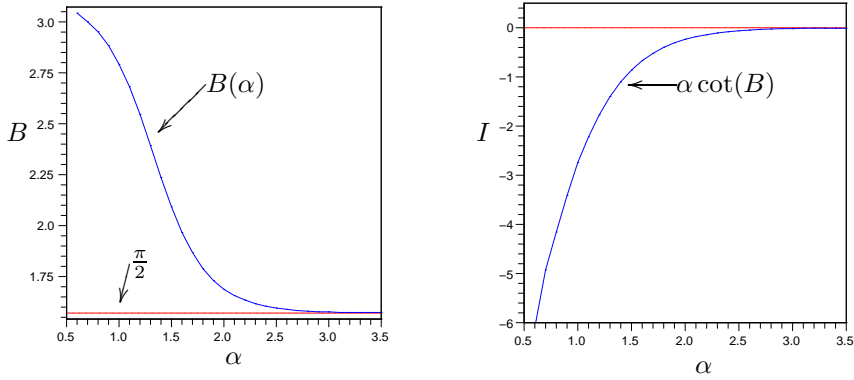
\begin{figure}[htb]
$$
\xymatrix{
\figa & \figb
}
$$
\caption{Left: Phase angle $B$ vs. $\alpha$; Right: $I=\alpha\cot(B)$
  vs. $\alpha$. Both plots use the solutions $z_j = Y_j$ of
  (\ref{eq:deY}) with $q(t) = 2\sech(t)^2$. These solutions are
  computed numerically in Maple using the $4$-th order Runge-Kutta
  method; the zeros are located by interval halving.} \label{fig:1}
\end{figure}
\end{center}
}

\end{document}